\documentclass[12pt,reqno]{amsart}
\usepackage{amsmath,amsfonts,amssymb,amsthm,amscd,latexsym}

\usepackage{color}

\newtheorem{theorem}{Theorem}[section]
\newtheorem{lemma}[theorem]{Lemma}

\newtheorem{corollary}[theorem]{Corollary}

\newtheorem{remark}[theorem]{Remark}

\newtheorem{ques}[theorem]{Question}

\newcommand{\cM}{{\mathcal M}}
\newcommand{\cN}{{\mathcal N}}
\newcommand{\cA}{{\mathcal A}}
\newcommand{\cB}{{\mathcal B}}

\newcommand{\cL}{{\mathcal L}}
\begin{document}

\date{\today}

\title[Ring isomorphisms of $\ast$-subalgebras of Murray--von Neumann factors]{Ring isomorphisms
of $\ast$-subalgebras of Murray--von Neumann factors}

\author[Sh. A. Ayupov]{Shavkat Ayupov}
\address{V.I.Romanovskiy Institute of Mathematics\\
  Uzbekistan Academy of Sciences\\ 81,  Mirzo Ulughbek street, 100170  \\
  Tashkent,   Uzbekistan}
\address{National University of Uzbekistan \\
4, University street, 100174, Tashkent, Uzbekistan}
\email{\textcolor[rgb]{0.00,0.00,0.84}{shavkat.ayupov@mathinst.uz}}

\author[K. Kudaybergenov]{Karimbergen Kudaybergenov}
\address{V.I.Romanovskiy Institute of Mathematics\\
  Uzbekistan Academy of Sciences \\ 81, Mirzo Ulughbek street, 100170  \\
  Tashkent,   Uzbekistan}
  \address{Department of Mathematics\\
 Karakalpak State University\\
 1, Ch. Abdirov, 230112,  Nukus, Uzbekistan}
\email{\textcolor[rgb]{0.00,0.00,0.84}{karim2006@mail.ru}}

\newcommand{\M}{\mathcal{M}}
\newcommand{\sm}{S(\mathcal{M})}

\begin{abstract}
The present paper is devoted to study of ring isomorphisms
of $\ast$-subalgebras of Murray--von Neumann factors.
Let $\cM,$  $\cN$ be von Neumann factors  of type II$_1,$
and let $S(\cM),$  $S(\cN)$ be the  $\ast$-algebras of all measurable operators affiliated with $\cM$ and $ \cN,$ respectively. Suppose that  $\cA\subset S(\cM),$ $\cB\subset S(\cN)$ are their $\ast$-subalgebras such that
$\cM\subset \cA,$ $\cN\subset \cB.$
We prove  that  for every ring isomorphism $\Phi: \cA \to  \cB$ there exist a positive invertible element
$a \in \cB$ with $a^{-1}\in \cB$ and a real
$\ast$-isomorphism $\Psi: \cM \to  \cN$ (which extends to a real
$\ast$-isomorphism from $\cA$ onto $\cB$) such that
$\Phi(x) = a\Psi(x)a^{-1}$ for
all $x \in  \cA.$ In particular, $\Phi$ is real-linear and continuous in the measure topology.
In particular, noncommutative Arens algebras and noncommutative $\cL_{log}$-algebras
associated with  von Neumann factors of type II$_1$ satisfy the above conditions and the main Theorem implies the automatic continuity of their ring isomorphisms in the corresponding metrics.
We also present an example of a $\ast$-subalgebra  in $S(\cM),$  which shows that the condition $\cM\subset \cA$ is essential in the above mentioned result.
\end{abstract}


\maketitle

\bigskip

\section{Introduction}

In 1930's, motivated by the
geometry of lattice of the projections  of type II$_1$ factors, von Neumann built the
theory on the correspondence between complemented orthomodular lattices
and regular rings. Let us recall one of his achievements \cite[Part II, Theorem 4.2]{Neu60},
applied to
the case of $\ast$-regular rings.
Let  $\mathfrak{R},$ $\mathfrak{R}'$ be  $\ast$-regular rings such that their  lattices of projections
$L_{\mathfrak{R}}$ and $L_{\mathfrak{R}'}$ are lattice-isomorphic. If
$\mathfrak{R}$ has order $n\ge 3$ (which means that it contains a ring of matrices of order $n$),
then there exists a ring isomorphism of  $\mathfrak{R}$ and $\mathfrak{R}'$ which generates given lattice isomorphism between
$L_{\mathfrak{R}}$ and $L_{\mathfrak{R}'}.$
One of  important classes of $\ast$-regular rings are
 the $\ast$-algebra  of operators affiliated with a finite von Neumann algebra.
Let $\cM$ be a von Neumann algebra and let $S(\cM)$ (respectively, $LS(\cM)$) be the $\ast$-algebra of all measurable (respectively, locally measurable) operators affiliated with  $\cM.$ Note that if $\cM$ is a finite von Neumann algebra then every operator affiliated with $\cM$ is automatically measurable and hence the $\ast$-algebras $S(\cM)$ and $LS(\cM)$ coincide. Applied to
the case of arbitrary type II$_1$ von Neumann algebras, the above von Neumann isomorphism theorem is formulated  as follows:
If  $\cM$ and $\cN$ are von Neumann algebras of type II$_1$ and
$\Phi:P(\cM) \to  P(\cN)$ is a lattice isomorphism then there exists a  unique ring isomorphism $\Psi :S(\cM) \to  S(\cN)$
such that $\Phi(l(x)) = l(\Psi(x))$ for any $x \in  S(\cM),$ where $l(a)$ is the left support of the element $a.$

Note that in the case of commutative regular rings the picture is completely different.
Let us recall a problem of isomorphisms for an important class of commutative regular rings with an atomic Boolean algebra of idempotents,
namely,  so-called Tychonoff semifields.
Given an arbitrary set  $\Delta$, a Tychonoff semifield $\mathbb{R}^\Delta$
 is defined as the product of $|\Delta|$ copies of the real field, equipped with the pointwise algebraic operations, natural partial order and the Tychonoff's topology. These operations make $\mathbb{R}^\Delta$  a topological regular  ring. The set of all idempotents of the semifield $\mathbb{R}^\Delta$ with the induced topology and order is topologically isomorphic to   $\{0,1\}^\Delta.$
For $g\in \Delta$ denote by $\mathbf{1}_g$ an atom from $\{0,1\}^\Delta$ defined as $\mathbf{1}_g(g)=1$
and $\mathbf{1}_g(g')=0$ for $g\neq g'$ ($g'\in \Delta$) and  $\mathbf{1}_\Delta$ is identity of $\mathbb{R}^\Delta.$

The following two questions are equivalent (see \cite{AB1970}, \cite{Ayupov77}):
\begin{itemize}
\item[(a)] Does there exist an algebraic homomorphism $\psi:\mathbb{R}^\Delta\to \mathbb{R}$
satisfying the condition $\psi(\mathbf{1}_g)=0$ for all $g\in \Delta,$ such that $\psi(\mathbf{1}_\Delta)=1?$
\item[(b)] Does there exist a non trivial two-valued countably additive measure $\mu:\{0,1\}^\Delta\to \mathbb{R}$
satisfying the condition $\mu(\mathbf{1}_g)=0$ for all $g\in \Delta?$
\end{itemize}
The second question is the famous Ulam Problem \cite{Ulam} which is connected with the properties of cardinal $|\Delta|.$

Returning to the noncommutative case recall that in the recent paper \cite{MMori2020} M. Mori  characterized lattice isomorphisms between projection lattices  $P(\cM)$ and $P(\cN)$  of arbitrary von Neumann algebras $\cM$   and $\cN$, respectively,  by means of ring isomorphisms between the algebras $LS(\cM)$ and $LS(\cN)$.  In this connection he investigated the following problem.

\begin{ques}\label{ques}
Let $\cM, \cN$ be von Neumann algebras. What is the general form of ring
isomorphisms from $LS(\cM)$ onto $LS(\cN)?$
\end{ques}

In \cite[Theorem B]{MMori2020} Mori himself gave  an answer to the above Question~\ref{ques} in the case of von Neumann algebras of type I$_\infty$ and III. Namely,
any ring isomorphism $\Phi$ from  $LS(\cM)$ onto $LS(\cN)$ has the form
$$
\Phi(x)=y\Psi(x)y^{-1},\, x\in LS(\cM),
$$
where $\Psi$ is a real $\ast$-isomorphism from $LS(\cM)$ onto $LS(\cN)$ and
$y\in LS(\cN)$ is an invertible element. Note that in the case  where $\Phi$ is an  algebraic isomorphism  of type I$_\infty$  von Neumann algebras, the above
presentation was obtained in \cite{AAKD11}.

In \cite{MMori2020} the author  conjectured that the  representation of ring isomorphisms, mentioned above for type  I$_\infty$ and III cases holds also for type II von Neumann algebras.
In \cite{AK2020} we have answered affirmatively to the above Question~\ref{ques} of Mori in the case of von Neumann algebras of type
II$_1.$  Namely, it was shown \cite[Theorems 1.3 and 1.4]{AK2020} that
if  $\cM,$  $\cN$ are  von Neumann algebras of type II$_1$   any
ring isomorphism $\Phi: S(\cM) \to  S(\cN)$ is continuous in the measure topology and  there exist an invertible element
$a \in S(\cN)$ and a real
$\ast$-isomorphism $\Psi: \cM \to  \cN$ (which extends to a real
$\ast$-isomorphism from $S(\cM)$ onto $S(\cN)$) such that
$\Phi(x) = a\Psi(x)a^{-1}$ for
all $x \in  S(\cM).$ As a corollary we also obtained that
for  von Neumann algebras $\cM$ and $\cN$   of type II$_1$ the projection lattices $P(\cM)$
and $P(\cN)$ are lattice isomorphic, if and only if the von Neumann algebras $\cM$ and $\cN$ are
Jordan $\ast$-isomorphic.
The present paper can be considered as an extention of the results from \cite{AK2020}.

In Section~2 we give preliminaries on Murray-von Neumann algebras and its special subalgebras -- so-called noncommutative Arens algebras and noncommutative $\cL_{log}$-algebras.

 The following Theorem which is  the main result of the present paper we shall prove in Section~3.

\begin{theorem}\label{ringisomorphism}
Let $\cM,$  $\cN$ be von Neumann factors  of type II$_1$ and let $\cA\subset S(\cM),$ $\cB\subset S(\cN)$ be $\ast$-subalgebras such that
$\cM\subset \cA,$ $\cN\subset \cB.$ Suppose that
$\Phi: \cA \to  \cB$ is a ring isomorphism.
Then there exist a positive invertible element
$a \in \cB$ with $a^{-1}\in \cB$ and a real
$\ast$-isomorphism $\Psi: \cM \to  \cN$ (which extends to a real
$\ast$-isomorphism from $\cA$ onto $\cB$) such that
$\Phi(x) = a\Psi(x)a^{-1}$ for
all $x \in  \cA.$ In particular, $\Phi$ is  real-linear and continuous in the measure topology.
\end{theorem}

In Section~4 we show that there is a $\ast$-regular subalgebra
of algebra of all measurable operators with respect to a
hyperfinite factor of type II$_1$ which admits an algebra automorphism, discontinuous in the measure topology.
In particular, since  Theorem~\ref{ringisomorphism} gives us automatic continuity in the measure topology of ring isomorphisms,
the mentioned example   shows that the condition $\cM\subset \cA$ is essential in Theorem~\ref{ringisomorphism}.

\section{Preliminaries}

For
$\ast$-algebras $\cA$  and
$\cB,$ a (not necessarily linear) bijection $\Phi: \cA \to  \cB$ is called
\begin{itemize}
\item a ring isomorphism if it is additive and multiplicative;
\item a real algebra isomorphism if it is a real-linear ring isomorphism;
\item an algebra isomorphism if it is a complex-linear ring isomorphism;
\item a real
$\ast$-isomorphism if it is a real algebra isomorphism and satisfies $\Phi(x^\ast) =
\Phi(x)^\ast$ for all $x \in  \cA;$
\item a
$\ast$-isomorphism if it is a complex-linear real $\ast$-isomorphism. 
\end{itemize}

\subsection{Murray-von Neumann algebra}

Let $H$  be a Hilbert space,  $B(H)$ be the $\ast$-algebra of all bounded linear operators
on $H$ and let $\cM$ be a von Neumann algebra in $B(H)$.

Denote by $P(\mathcal{M})$ the set of all projections in $\mathcal{M}.$ Recall that two projections $e, f \in  P(\mathcal{M})$ are called \textit{equivalent}  (denoted as $e\sim f$) if there exists an element
$u \in \mathcal{M}$ such that $u^\ast  u = e$ and $u u^\ast  = f.$
For projections $e, f \in  \mathcal{M}$
notation $e \precsim  f$ means that there exists a projection $q \in  \mathcal{M}$ such that
$e\sim q \leq f.$ A projection $p \in \mathcal{M}$ is said to be \textit{finite}, if it is not equivalent to its proper sub-projection, i.e.
the conditions $q \leq  p$ and $q\sim p$ imply that $q = p$ (for details information concerning von Neumann algebras see \cite{KRII, Sakai_book}).

A densely defined closed linear operator $x : \textrm{dom}(x) \to  H$
(here the domain $\textrm{dom}(x)$ of $x$ is a dense linear subspace in $H$) is said to be \textit{affiliated} with $\mathcal{M}$
if $yx \subset  xy$ for all $y$ from the commutant $\mathcal{M}'$  of the algebra $\mathcal{M}.$

A linear operator $x$ affiliated with $\mathcal{M}$ is called \textit{measurable} with respect to $\mathcal{M}$ if
$e_{(\lambda,\infty)}(|x|)$ is a finite projection for some $\lambda>0.$ Here
$e_{(\lambda,\infty)}(|x|)$ is the  spectral projection of $|x|$ corresponding to the interval $(\lambda, +\infty).$
We denote the set of all measurable operators by $S(\mathcal{M}).$

Let $x, y \in  S(\mathcal{M}).$ It is well known that $x+y$ and
$xy$ are densely-defined and preclosed
operators. Moreover, the (closures of) operators $x + y, xy$ and $x^\ast$  are also in $S(\mathcal{M}).$
When
equipped with these operations, $S(\mathcal{M})$ becomes a unital $\ast$-algebra over $\mathbb{C}$  (see \cite{MC,Segal}). It
is clear that $\mathcal{M}$  is a $\ast$-subalgebra of $S(\mathcal{M}).$
In the case of finite von Neumann algebra $\cM$, all operators affiliated with $\cM$ are measurable and the algebra $S(\cM)$ is referred to as the \emph{Murray-von Neumann algebra} associated with $\cM$ (see \cite{KL}).

Let $\cM$ be a von Neumann algebra with a faithful normal finite trace $\tau.$
Consider the topology  $t_\tau$ of convergence in measure or \textit{measure topology} \cite{Nel}
on $S(\mathcal{M}),$ which is defined by
the following neighborhoods of zero:
$$
N(\varepsilon, \delta)=\{x\in S(\mathcal{M}): \exists \, e\in P(\mathcal{M}), \, \tau(\mathbf{1}-e)\leq\delta, \, xe\in
\mathcal{M}, \, \|xe\|_\cM\leq\varepsilon\},
$$
where $\varepsilon, \delta$
are positive numbers. The pair $(S(\cM), t_\tau)$ is a complete topological $\ast$-algebra.

We define the so-called rank  metric $\rho$ on $S(\mathcal{M})$  by setting
$$
\rho(x, y)=\tau( r(x-y))=\tau(l(x-y)),\,\, x, y\in S(\cM).
$$
In fact, the rank-metric $\rho$ was firstly introduced in a general case of regular rings in \cite{Neu37}, where it was shown it is a metric.
By \cite[Proposition 2.1]{Ciach}, the algebra $S(\mathcal{M})$ equipped with the metric $\rho$ is a complete topological $*$-ring.

Let $\cM$ be a finite von Neumann algebra. A $\ast$-subalgebra  $\mathcal{A}$ of $S(\mathcal{M})$ is said to be
\emph{regular}, if it is a regular ring in the sense of von Neumann, i.e., if for every
$a\in\mathcal{A}$ there exists an element  $b\in\mathcal{A}$ such that
$aba=a.$

Given $a\in S(\cM)$  let  $a=v|a|$ be the polar decomposition of $a.$
Then $l(a) =v v^\ast$ and  $r(a)=v^\ast v$ are left and right supports of the element  $a$, respectively.
The projection  $s(a)=l(a)\vee r(a)$ is the support of the element $a$. It is clear that $r(a)=s(|a|)$ and  $l(a)=s(|a^*|)$.
There is a unique element $i(a)$ in
$S(\mathcal{M})$ such that $ai(a)=l(a),\ i(a)a=r(a),\ ai(a)a=a,$
$i(a)l(a)=i(a)$ and  $r(a)i(a)=i(a).$
The element  $i(a)$ is called the \emph{partial inverse} of the element $a.$
Therefore  $S(\mathcal{M})$ is a regular $*$-algebra  (see \cite{Berber}, \cite{Saito}).

Let $e\in S(\cM)$ be an idempotent, i.e., $e^2=e.$  Recall the following properties of the left projection \cite{AK2020}:
\begin{equation}\label{leftpr}
l(e)e=e,\,\, el(e)=l(e).
\end{equation}
It is clear that the left support $l(e)$ of the idempotent $e$ is uniquely determined by the above two equalities.

\subsection{Noncommutative Arens algebras and noncommutative $\cL_{log}$-algebras}

\

In this subsection we present two classes of $\ast$-subalgebras  in $S(\cM)$ which satisfy the conditions of Theorem~\ref{ringisomorphism}.

Let $\cM$ be a von Neumann algebra with a faithful normal semifinite trace $\tau.$
Given $p \ge 1$ denote by $L_p(\cM,\tau)$ the set of all elements $x$ from $S(\cM)$ such that
$
\tau(|x|^p)<+\infty.$
It is well-known that $L_p(\cM,\tau)$ is a
Banach space with respect to the norm
$$
||x||_p=\left(\tau\left(|x|^p\right)\right)^{1/p},\, x \in  L_p(\cM,\tau).
$$
The intersection
$$
L^\omega(\cM,\tau) = \bigcap\limits_{p\ge 1}L_p(\cM,\tau).
$$
is a $\ast$-subalgebra in $S(\cM)$  \cite{Inoue}.
The algebra  $L^\omega(\cM,\tau)$ is called a noncommutative Arens algebra and it
is a locally convex complete metrizable
$\ast$-algebra with respect to the topology  generated by the family of norms $\{||\cdot||_p\}_{p\ge 1}$  (see(\cite{ARZ, AAK07}).
Note that  in the commutative (functional space) case the algebra $L^\omega [0, 1]$  was introduced by R. Arens in \cite{Arens}.

Let $\cL_{log}(\cM,\tau)$ be the set of all elements $x$ from $S(\cM)$ which satisfy
$$
||x||_{log}=\tau\left(\log(\mathbf{1} + |x|)\right) <+\infty.
$$
It is known that \cite[Theorem 4.9]{DSZ16} the pair $\left(\cL_{log}(\cM,\tau), ||\cdot||_{log}\right)$ is a topological $\ast$-algebra with respect to a complete metric space topology.

Note  that by \cite[Proposition 4.7]{DSZ16}, it follows that the Arens algebra
$L^\omega(\cM, \tau)$ is $\ast$-subalgebra of
$\cL_{log}(\cM, \tau).$
It clear that if  $\tau$ is  finite,  then $\cM$ is a $\ast$-subalgebra in both algebras $L^\omega(\cM, \tau)$ and
$\cL_{log}(\cM, \tau).$  It should be noted that  $\cL_{log}$-algebras were considered as certain invariant
subspaces in order to obtain upper-triangular-type decompositions of unbounded operators (see \cite{DSZ2017}).

\section{Proof of the main result}

Let   $\mathcal{M}$ and $\cN$ be  arbitrary type II$_1$ von Neumann factors with  faithful normal normalised
traces $\tau_{\cM}$ and
$\tau_{\cN},$ respectively and let $\Phi:\cA \to \cB$ be a ring isomorphism.

The following is a well-known result which is crucial in our construction.
For convenience, we include the full proof.

By  $\rho_\cM$ we denote  the rank-metric on $S(\cM).$

\begin{lemma}\label{dense}
$\cM$ is $\rho_{\cM}$-dense in $S(\cM).$
\end{lemma}

\begin{proof}
Let  $x\in S(\cM)$  and let  $x=v|x|$ be the polar decomposition of $x.$ Consider the spectral resolution
$|x|=\int\limits_0^\infty \lambda de_\lambda$ of $|x|$ and let $e_n=e_{(0,n]}(|x|)$ be the  spectral projection of $|x|$ corresponding to the interval $(0, n].$ Set $x_n=xe_n,$ $n\in \mathbb{N}.$
Since  $|x|e_n\le n e_n\in \cM,$ it follows that
$x_n=v|x|e_n\in \cM$ for all $n\in \mathbb{N}.$ Further,
\begin{align*}
\rho_{\cM}(x,x_n) &= \tau\left(r(x-x_n)\right)=\tau\left(r(v|x|(\mathbf{1}-e_n))\right)\le\tau\left(\mathbf{1}-e_n\right)\to 0,
\end{align*}
because $e_n\uparrow \mathbf{1}.$ This means that $x_n\stackrel{\rho_{\cM}}\longrightarrow x,$ and therefore
$\cM$ is $\rho_{\cM}$-dense in $S(\cM).$
\end{proof}

In the proof of the next Lemma by $t_\cM$ and $t_\cN$ we denote the measure topologies on $S(\cM)$ and $S(\cN),$ respectively.

\begin{lemma}\label{rhocon}
$\Phi$ is continuous in the topology generated by the rank-metric.
\end{lemma}

\begin{proof}
Consider the mapping
\begin{align}\label{latiso}
p\in P(\cM) \mapsto l\left(\Phi(p)\right)\in P(\cN),
\end{align}
where $l(x)$ is the left support of the element $x.$

Let us show that this mapping is a lattice-isomorphism from $P(\cM)$ onto $P(\cN).$

For $p, q\in P(\cM)$ with $p\le q$ we have that
\begin{align*}
l\left(\Phi(p)\right) & =l\left(\Phi(qp)\right)=l\left(\Phi(q)\Phi(p)\right)\le l\left(\Phi(q)\right).
\end{align*}

Let $p, q\in P(\cM)$ be projections such that  $l\left(\Phi(p)\right)\le l\left(\Phi(q)\right).$
Then
\begin{align*}
\Phi(p) & =l\left(\Phi(p)\right)\Phi(p)=l\left(\Phi(q)\right)l\left(\Phi(p)\right)\Phi(p)\stackrel{\eqref{leftpr}}=
\Phi(q)l\left(\Phi(q)\right)l\left(\Phi(p)\right)\Phi(p)\\
&=\Phi(q)l\left(\Phi(p)\right)\Phi(p)=\Phi(q)\Phi(p)=\Phi(qp).
\end{align*}
Since $\Phi$ is a bijection, it follows that $p=qp,$ i.e. $p\le q.$ In particular, if
$l\left(\Phi(p)\right)=l\left(\Phi(q)\right),$ where  $p, q\in P(\cM),$   then
$p=q.$

Let $f\in P(\cN).$ Take an element $x\in S(\cM)$ such that $\Phi(x)=f.$ Then $x$ is an idempotent, and hence
$x=e+w,$ where $e=l(x)$ and $w\in eS(\cM)(\mathbf{1}-e).$
We have that
\begin{align*}
\Phi(e) & =\Phi(x)-\Phi(w)=f-\Phi(e)\Phi(w)(\mathbf{1}-\Phi(e)).
\end{align*}
Therefore
\begin{align*}
\Phi(e)f & =f^2-\Phi(e)\Phi(w)(\mathbf{1}-\Phi(e))f=f-\Phi(e)\Phi(w)(\mathbf{1}-\Phi(e))\Phi(x)\\
&=f-\Phi(e)\Phi(w)(\Phi(x)-\Phi(ex))=f-\Phi(e)\Phi(w)(\Phi(x)-\Phi(x))=f.
\end{align*}
Further,
\begin{align*}
f\Phi(e) & =\Phi(x)\Phi(e)=\Phi(xe)=\Phi((e+w)e)=\Phi(e),
\end{align*}
because $we=w(\mathbf{1}-e)e=0.$ The above two equalities show $l\left(\Phi(e)\right)=f.$

So, the mapping defined by \eqref{latiso} is an order-isomorphism from $P(\cM)$ onto $P(\cN).$
By \cite[Page 24, Lemma 2]{Birkhoff}, this mapping is a lattice-isomorphism from $P(\cM)$ onto $P(\cN),$
that is,
\begin{align*}
l\left(\Phi(p\vee q)\right)=l\left(\Phi(p)\right)\vee l\left(\Phi(q)\right),\,\,
l\left(\Phi(p\wedge q)\right)=l\left(\Phi(p)\right)\wedge l\left(\Phi(q)\right)
\end{align*}
for all $p,q\in P(\cM).$

Since $S(\cM)$ and $S(\cN)$ are regular rings containing the matrix ring over the field of complex numbers of order bigger than $3,$ by \cite[Part II, Theorem 4.2]{Neu60}, the lattice isomorphism of $P(\cM)$ and $P(\cN)$ defined as \eqref{latiso} is generated by a ring isomorphism $\Theta$  from $S(\cM)$ onto $S(\cN),$ i.e., $l\left(\Theta(p)\right)=l\left(\Phi(p)\right)$
for all $p\in P(\cM).$
By \cite[Theorem  1.3]{AK2020} the ring isomorphism $\Theta$ is continuous in the measure topology.

Let $x_n\stackrel{\rho_{\cM}}\longrightarrow 0.$
Then $\tau_\cM(l(x_n))=\rho_\cM(x_n, 0)\rightarrow 0,$ and therefore $l(x_n)\stackrel{t_{\cM}}\longrightarrow 0.$
Since
$\Theta$ is continuous in the measure topology, it follows that
$\Theta(l(x_n))\stackrel{t_{\cN}}\longrightarrow 0.$
By \cite[Lemma 2.2]{AK2020}, it follows that $l\left(\Theta(l(x_n))\right)\stackrel{t_{\cN}}\longrightarrow 0,$
 because $\Theta(l(x_n))$ is an idempotent for all $n.$
Further, we have that
\begin{align*}
l\left(\Phi(x_n)\right) & =l\left(\Phi(l(x_n)x_n)\right)=l\left(\Phi(l(x_n))\Phi(x_n)\right)\le l\left(\Phi(l(x_n))\right)\\
& =l\left(\Theta(l(x_n))\right)\stackrel{t_{\cN}}\longrightarrow 0.
\end{align*}
Thus $\rho_\cN\left(\Phi(x_n),0\right)=\tau_\cN\left(l\left(\Phi(x_n)\right)\right)\to 0.$ This means that $\Phi(x_n)\stackrel{\rho_{\cN}}\longrightarrow 0.$
\end{proof}

\begin{lemma}\label{ext}
A ring isomorphism $\Phi:\cA\to \cB$  extends to a ring isomorphism $\widetilde{\Phi}$ from $S(\cM)$ onto $S(\cN).$
\end{lemma}

\begin{proof} By Lemma~\ref{dense} the $\ast$-subalgebra $\cM$ is $\rho_\cM$-dense in $S(\cM),$ and therefore $\cA$ is also
  $\rho_\cM$-dense in $S(\cM).$ Using this observation we can define a mapping
$\widetilde{\Phi}$ from $S(\cM)$ into $S(\cN)$ as
\begin{align}\label{extring}
\widetilde{\Phi} & =\rho_\cN-\lim\limits_{n\to\infty}\Phi(x_n),
\end{align}
where $\{x_n\}\subset \cA$ is a sequence such that $x_n\stackrel{\rho_{\cM}}\longrightarrow x.$

Let us show the mapping $\widetilde{\Phi}$ is a well-defined ring isomorphism.

Firstly, we shall show that this mapping is well-defined.
Let  $\{x_n\}\subset \cA$ be  a sequence such that $x_n\stackrel{\rho_{\cM}}\longrightarrow x.$
Then $x_n-x_m\stackrel{\rho_{\cM}}\longrightarrow 0$ as $n, m\to\infty.$
Since by Lemma~\ref{rhocon}, $\Phi$ is continuous in the topology generated by the rank-metric, it follows that
$\Phi(x_n)-\Phi(x_m)\stackrel{\rho_{\cN}}\longrightarrow 0$ as $n, m\to\infty.$
Since $S(\cN)$ is $\rho_\cN$-complete, it follows that there exists
$\rho_\cN-\lim\limits_{n\to\infty}\Phi(x_n)\in S(\cN).$
So, the limit on the right side of \eqref{extring} exists.

Now suppose that  $\{x_n\}, \{x'_n\}\subset \cA$ are sequences such that $x_n\stackrel{\rho_{\cM}}\longrightarrow x$ and
\linebreak $x'_n\stackrel{\rho_{\cM}}\longrightarrow x.$
Then $x_n-x'_n\stackrel{\rho_{\cM}}\longrightarrow 0$ as $n\to\infty,$
and again by $\rho_\cM$-$\rho_\cN$-continuity of $\Phi$ we obtain that
$\Phi(x_n)-\Phi(x'_n)\stackrel{\rho_{\cN}}\longrightarrow 0$ as $n\to\infty.$
Thus
$$
\widetilde{\Phi}(x)=\rho_\cN-\lim\limits_{n\to\infty}\Phi(x_n)=\rho_\cN-\lim\limits_{n\to\infty}\Phi(x'_n).
$$
So, $\widetilde{\Phi}$ is a well-defined mapping.

Let us show the additivity and multiplicativity of $\widetilde{\Phi}.$ For $x, y\in S(\cM)$ take sequences
$\{x_n\}, \{y_n\}\subset \cA$ such that
$x_n\stackrel{\rho_{\cM}}\longrightarrow x$ and
$y_n\stackrel{\rho_{\cM}}\longrightarrow y.$ Then
\begin{align*}
\widetilde{\Phi}(x) + \widetilde{\Phi}(y) & =\lim\limits_{n\to\infty}\Phi(x_n)+\lim\limits_{n\to\infty}\Phi(x_n)=
\lim\limits_{n\to\infty}\Phi(x_n+y_n)=\widetilde{\Phi}(x+y).
\end{align*}
By a similar argument we get the  multiplicativity of $\widetilde{\Phi}.$

The next step is the proof of the surjectivity of $\widetilde{\Phi}.$

Let $y\in S(\cN)$ and let $\{y_n\}\subset \cB$ be a sequence such that $y_n\stackrel{\rho_{\cN}}\longrightarrow y.$
Then $y_n-y_m\stackrel{\rho_{\cN}}\longrightarrow 0$ as $n, m\to\infty.$
Since by Lemma~\ref{rhocon}, $\Phi^{-1}$ is $\rho_\cN$-$\rho_\cM$-continuous, it follows that
$\Phi^{-1}(y_n)-\Phi^{-1}(y_m)\stackrel{\rho_{\cM}}\longrightarrow 0$ as $n, m\to\infty.$
Since $S(\cM)$ is $\rho_\cM$-complete, it follows that there exists
$x=\rho_\cM-\lim\limits_{n\to\infty}\Phi^{-1}(y_n)\in S(\cM).$
Then
$$
\widetilde{\Phi}(x)=\rho_\cN-\lim\limits_{n\to\infty}\Phi(\Phi^{-1}(y_n))=\rho_\cN-\lim\limits_{n\to\infty}y_n=y.
$$

The final  step of the proof is the injectivity  of $\widetilde{\Phi}.$

Let $x\in S(\cM)$ and suppose that $\widetilde{\Phi}(x)=0.$
Let  $x=v|x|$ be the polar decomposition of $x.$ Consider the spectral resolution
$|x|=\int\limits_0^\infty \lambda de_\lambda$ of $|x|$ and let $e_n=e_{(0,n]}(|x|)$ be the  spectral projection of $|x|$ corresponding to the interval $(0, n].$
Note that $xe_n\in  \cM\subset \cA$ for all $n\in \mathbb{N}.$ Further,
\begin{align*}
0 & = \widetilde{\Phi}(x)\widetilde{\Phi}(e_n)=\widetilde{\Phi}(xe_n)=\Phi(xe_n).
\end{align*}
Since $\Phi$ is a ring isomorphism, it follows that $xe_n=0$ for all $n\in \mathbb{N}.$
From $e_n\uparrow \mathbf{1},$ we have that  $xe_nx^\ast \uparrow x x^\ast.$ Thus $xx^\ast=0,$ and hence $x=0.$ The proof is complete.
\end{proof}

\begin{proof}[Proof of Theorem \ref{ringisomorphism}]
By Lemma~\ref{ext}
a ring isomorphism $\Phi:\cA\to \cB$  extends to a ring isomorphism $\widetilde{\Phi}$ from $S(\cM)$ onto $S(\cN).$
Then by \cite[Theorem 1.4]{AK2020}  there exist an invertible element
$a \in S(\cN)$ and a real
$\ast$-isomorphism $\Psi: \cM \to  \cN$ (which extends to a real
$\ast$-isomorphism from $S(\cM)$ onto $S(\cN)$) such that
$\widetilde{\Phi}(x) = a\Psi(x)a^{-1}$ for
all $x \in  S(\cM).$

Let us first to show that $a\in \cB$ and $a^{-1}\in\cB.$

Let    $a=v|a|$ be the polar decomposition of $a.$ Since $a$ is invertible, it follows that $v$ is unitary.
Since
$$
a\Psi(x)a^{-1}=v|a|v^\ast v\Psi(x)v^\ast v|a|^{-1}v^\ast,
$$
 replacing, if necessary, $a$ to $v|a|v^\ast$ and $\Phi$ to $v\Psi(\cdot)v^\ast,$ we can assume that $a$ is a positive invertible element in
$S(\cN).$

Consider the spectral resolution
$a=\int\limits_0^\infty \lambda de_\lambda$ of $a$
and let $e_\lambda=e_{(0,\lambda]}(a)$ be the  spectral projection of $a$ corresponding to the interval $(0, \lambda],$ $\lambda>0.$

If $a\in \cN,$ then $a\in \cB,$ because $\cN\subset \cB.$ So, we need to consider the case $a\in S(\cN)\setminus \cN.$
Then there exists a positive number $\lambda$ such that
$\tau_\cN(\mathbf{1}-e_{\lambda})\le \frac{1}{2}.$ Since $\tau_\cM$ is a normalised trace, it follows that
\begin{align*}
\tau_\cN(e_{\lambda}) & \ge\frac{1}{2}\ge \tau_\cN(\mathbf{1}-e_{\lambda}).
\end{align*}
Since $\cN$ is a von Neumann factor of type II$_1,$ it follows that
$\mathbf{1}-e_{\lambda}\precsim e_{\lambda}.$ Take a partial isometry
$u\in \cN$ such that $u u^\ast=\mathbf{1}-e_{\lambda}$ and $u^\ast u \le e_{\lambda}.$
By the choice of the spectral projection $e_{\lambda}$ we obtain that
\begin{align}\label{lambda12}
ae_{\lambda}\le \lambda e_{\lambda}.
\end{align}
Further, let  $w$ be an element in $\cA$ such that $\Psi(w)=u.$ Note that  $w$ is also a partial isometry in $\cM\subset \cA.$ We have that
\begin{align*}
\Phi(w) & = \widetilde{\Phi}(w) = a\Psi(w)a^{-1}= a u a^{-1}\\
& = a (u u^\ast) u a^{-1}=a(\mathbf{1}-e_{\lambda}) u a^{-1}.
\end{align*}
Multiplying  the last equality from the right side by the element $a e_{\lambda}u^\ast$ we obtain
\begin{align*}
\Phi(w) a e_{\lambda}u^\ast & = \Big(a(\mathbf{1}-e_{\lambda}) u a^{-1}\Big)a e_{\lambda}u^\ast=a(\mathbf{1}-e_{\lambda}) u e_{\lambda}u^\ast\\
&=a(\mathbf{1}-e_{\lambda}) u u^\ast u e_{\lambda}u^\ast=a(\mathbf{1}-e_{\lambda}) u u^\ast u u^\ast=a(\mathbf{1}-e_{\lambda}),
\end{align*}
because $u u^\ast=\mathbf{1}-e_{\lambda}$ and $u^\ast u \le e_{\lambda}.$ Taking into account \eqref{lambda12} and the inclusions  $u^\ast, \Phi(w)\in \cB,$ from the last equality, we conclude that
$a(\mathbf{1}-e_{\lambda})\in \cB.$
Hence
$$
a=ae_{\lambda}+a(\mathbf{1}-e_{\lambda})\in \cN+\cB\subset \cB.
$$
By a similar argument we can show that $a^{-1}\in \cB.$

Finally, we
show that the restriction $\Psi|_\cA$ of the real $\ast$-isomorphism $\Psi$ onto $\cA$ maps $\cA$ onto $\cB.$
Indeed, since $a, a^{-1}\in \cB,$ it follows that
$$
\Psi(x)=a^{-1}\widetilde{\Phi}(x)a=a^{-1}\Phi(x)a\in \cB
$$
for all $x\in \cA.$ Further, considering the inverse map $\Phi^{-1}$ which  acts as
$$
\Phi^{-1}(y) = \Psi^{-1}\left(a^{-1}\right)\Psi^{-1}(x)\Psi^{-1}(a),\,\,y \in  \cB,
$$
we conclude that both $\Psi^{-1}(a)$ and $\Phi^{-1}(a)^{-1}$  are in $\cA,$ and that  $\Psi^{-1}$ maps $\cB$ onto $\cA.$ So,
$\Psi(\cA)=\cB.$
The proof is complete.
\end{proof}

\begin{corollary}\label{metralg}
Let $\cA$ and $\cB$ be $\ast$-subalgebras from Theorem~\ref{ringisomorphism}. Suppose that these algebras are equipped with  metrics
$\rho_\cA$ and $\rho_\cB$  respectively, such that both $\left(\cA,\rho_\cA\right),$  $\left(\cB,\rho_\cB\right)$ are complete topological $\ast$-algebras and additionally, convergence with respect their metrics implies the convergence in measure. Then any ring isomorphism $\Phi:\cA\to \cB$ is $\rho_\cA$-$\rho_\cB$-continuous.
\end{corollary}

\begin{proof}
Take a sequence $\{x_n\}\subset \cA$ such that $x_n\stackrel{\rho_\cA}{\longrightarrow}0$
 and  $\Phi(x_n)\stackrel{\rho_\cB}{\longrightarrow} y\in \cB,$ in particular, $\Phi(x_n)\rightarrow y$ in the measure topology in $S(\cN).$ Since $\rho_\cA$-convergence implies the convergence in  measure  in $S(\cM),$
 it follows that $x_n\rightarrow 0$ in the measure topology. Further, the continuity of $\Phi$ in the measure topology implies
that $\Phi(x_n)\rightarrow 0$ in the measure topology in $S(\cN).$ Thus $y=0,$ and hence by the closed graph theorem (see
\cite[Page 79]{Yos}), we conclude that $\Phi$ is $\rho_\cA$-$\rho_\cB$-continuous.
\end{proof}

\begin{remark}
Note that the noncommutative Arens algebras and noncommutative $\cL_{log}$-algebras
associated with  von Neumann factors of type II$_1$ satisfy the  conditions of Corollary~\ref{metralg}.
Indeed, in the following series identical imbeddings are continuous
\begin{align*}
\left(L^\omega(\cM, \tau),\{||\cdot||_p\}_{p\ge 1}\right)\subset \left(L_1(\cM, \tau),||\cdot||_1\right)\subset\left(L_{log}(\cM, \tau),||\cdot||_{log}\right)
\subset\left(S(\cM), t_\tau\right).
\end{align*}
The  continuity of the first imbedding  immediately follows from the definition, the continuity of the  second and the third imbeddings follow from \cite[Proposition 4.7 and Remark 4.8]{DSZ16}.
\end{remark}

\section{Discontinuous algebra automorphism  of  a $\ast$-regular algebra}

In this Section we show that there is a $\ast$-regular subalgebra
of algebra of all measurable operators with respect to the
hyperfinite factor of type II$_1$ which admits an algebra automorphism, discontinuous in the measure topology.

Let $\mathcal{R}$ be the hyperfinite  II$_1$-factor  with the faithful normal normalised trace $\tau.$
There is  a  system of matrix units
$\mathcal{E}=\left\{e^{(n)}_{ij}:\  n=0,1,\dots, \ i,j=1,\dots,2^n\right\}$ in  $\mathcal{R}$
(here $e^{(0)}_{1,1}=\mathbf{1}$) such that \cite{Tau}
\begin{itemize}
\item[(a)] $e^{(n)}_{ij}e^{(n)}_{k,l}=\delta_{jk}e^{(n)}_{il};$
\item[(b)] $\left(e^{(n)}_{ij}\right)^\ast=e^{(n)}_{ji};$
\item[(c)] $\sum\limits_{i=1}^{2^n}e^{(n)}_{ii}=\mathbf{1};$
\item[(d)] $e^{(n)}_{ij}=e^{(n+1)}_{2i-1,2j-1}+e^{(n+1)}_{2i,2j}.$
\end{itemize}

For any  $n=0,1, \ldots$ denote by $\mathcal{R}_n$ the $\ast$-subalgebra of
$\mathcal{R}$, generated by the system of matrix units $\left\{e^{(n)}_{ij}:\ i,j=1,\dots,2^n\right\}.$
Then
$$
\mathcal{R}_0\subset \mathcal{R}_1\subset\dots\subset\mathcal{R}_n\subset\dots
$$
and each $\ast$-subalgebra $\mathcal{R}_n$ is $\ast$-isomorphic to the algebra of all $2^n\times 2^n$-matrices over the field $\mathbb{C}.$ Set
$$
\mathcal{R}_\infty=\bigcup_{n=1}^\infty\mathcal{R}_n.
$$
Then $\mathcal{R}_\infty$ is a $\ast$-regular algebra as a sum of  increasing  sequence of matrix algebras
(see e.g. \cite[Theorem  3]{Skor}).

Now we begin to construct a discontinuous  algebra automorphism of $\mathcal{R}_\infty.$

Define  the sequences $\left\{a_n: n=1,2, \ldots\right\}$ and $\left\{c_n: n=1,2, \ldots\right\}$ $(a_n, c_n\in \mathcal{R}_n)$ by the   rule
\begin{eqnarray*}\label{c_n}
&c_n & = 2^n\sum\limits_{k=1}^{2^{n-1}} e_{2k-1,2k-1}^{(n)}+\sum\limits_{k=1}^{2^{n-1}} e_{2k,2k}^{(n)}, \, n\in \mathbb{N}
\end{eqnarray*}
and
\begin{eqnarray*}\label{c_n}
&a_n & = \prod\limits_{k=1}^{n} c_k, \, n\in \mathbb{N}.
\end{eqnarray*}
Note that all $a_n, c_n$ are invertible in $\mathcal{R}_n.$

For $n\geq 1$ define an algebra automorphism $\Phi_n$ of $\mathcal{R}_n$ as follows
$$
\Phi_n(x)=a_n x a_n^{-1}, \, x\in \mathcal{R}_n.
$$

\begin{lemma}\label{deltan-1}
$\Phi_n|_{\mathcal{R}_{n-1}}=\Phi_{n-1}$  for all $n\in \mathbb{N}.$
\end{lemma}

\begin{proof}
Since $a_{n-1}^{-1}a_n=a_na_{n-1}^{-1}=c_n,$ it suffices to  show that
$$
\left[c_n, \mathcal{R}_{n-1}\right]=0.
$$
Using the property (d) of matrix units, for fixed $1\leq i, j \leq 2^{n-1}$ we have that
\begin{eqnarray*}
\left[c_n, e^{(n-1)}_{i, j}\right] & \stackrel{(d)}= & \left[2^n\sum\limits_{k=1}^{2^{n-1}} e_{2k-1,2k-1}^{(n)}+
\sum\limits_{k=1}^{2^{n-1}} e_{2k,2k}^{(n)}, e^{(n)}_{2i-1, 2j-1}+e^{(n)}_{2i, 2j}\right] \\
& = & 2^n\left[\sum\limits_{k=1}^{2^{n-1}} e_{2k-1,2k-1}^{(n)}, e^{(n)}_{2i-1, 2j-1}+e^{(n)}_{2i, 2j}\right]+\left[\sum\limits_{k=1}^{2^{n-1}} e_{2k,2k}^{(n)}, e^{(n)}_{2i-1, 2j-1}+e^{(n)}_{2i, 2j}\right]\\
&=& 2^ne_{2i-1,2i-1}^{(n)}e^{(n)}_{2i-1, 2j-1}-2^ne^{(n)}_{2i-1, 2j-1}e^{(n)}_{2j-1, 2j-1}+ e_{2i,2i}^{(n)}e^{(n)}_{2i, 2j}-e^{(n)}_{2i, 2j}e^{(n)}_{2j, 2j}=0.
\end{eqnarray*}
\end{proof}

\begin{lemma}
There exists an algebra isomorphism $\Phi:\mathcal{R}_\infty\to \mathcal{R}_\infty$ such that $\Phi|_{\mathcal{R}_{n}}=\Phi_{n}$ for all $n=1,2,\ldots.$
\end{lemma}

\begin{proof}
Define the mapping $\Phi:\bigcup_{n=1}^\infty\mathcal{R}_n\to \bigcup_{n=1}^\infty\mathcal{R}_n$ by setting $\Phi|_{\mathcal{R}_n}=\Phi_n.$ By Lemma~\ref{deltan-1} we have  $\Phi_{n}|_{\mathcal{R}_{n-1}}=\Phi_{n-1}$, and therefore, $\Phi$ is a well-defined mapping. It is clear that $\Phi$ is an algebra automorphism of  $\mathcal{R}_\infty.$
\end{proof}

\begin{lemma}
For each $n\geq 1$ the element $a_n$ can be represented as
\begin{eqnarray}\label{inveran}
a_n=\sum\limits_{k=1}^{2^n}\gamma_k^{(n)}e_{k,k}^{(n)}
\end{eqnarray}
where
\begin{eqnarray}\label{gamma}
\gamma_{2k-1}^{(n)}=2^n \gamma_{2k}^{(n)}
\end{eqnarray}
for all  $n\geq1$ and $k=1, \ldots, 2^{n-1}.$
\end{lemma}

\begin{proof}
The proof is by the induction on $n.$ For $n=1$ we have that
\begin{eqnarray*}
a_1 & = & 2^1 e_{1,1}^{(1)}+e_{2,2}^{(1)},
\end{eqnarray*}
and therefore
$
\gamma_1^{(1)}=2\gamma_{2}^{(1)}.
$
Suppose that we have proved  the required assertion  for $n-1.$ Taking into account that $a_n=a_{n-1}c_n$ and the equality $e^{(n-1)}_{k,k}=e^{(n)}_{2k-1,2k-1}+e^{(n)}_{2k,2k}$ we have that
\begin{eqnarray*}
&a_n & = a_{n-1}c_n=\sum\limits_{l=1}^{2^{n-1}}\gamma_l^{(n-1)}e_{l,l}^{(n-1)}\Big(2^n\sum\limits_{k=1}^{2^{n-1}} e_{2k-1,2k-1}^{(n)}+\sum\limits_{k=1}^{2^{n-1}} e_{2k,2k}^{(n)}\Big)\\
&=& \Big(\sum\limits_{l=1}^{2^n} \gamma_l^{(n-1)}e_{2l-1,2l-1}^{(n)}+\gamma_l^{(n-1)}e_{2l,2l}^{(n)}\Big)\Big(2^n\sum\limits_{k=1}^{2^{n-1}} e_{2k-1,2k-1}^{(n)}+\sum\limits_{k=1}^{2^{n-1}} e_{2k,2k}^{(n)}\Big)\\
&=& \sum\limits_{k=1}^{2^n}\Big(2^n\gamma_k^{(n-1)}e_{2k-1,2k-1}^{(n)}+\gamma_k^{(n-1)}e_{2k,2k}^{(n)}\Big).
\end{eqnarray*}
Thus
\begin{eqnarray*}
a_n=\sum\limits_{k=1}^{2^n}\gamma_k^{(n)}e_{k,k}^{(n)},
\end{eqnarray*}
where
\begin{eqnarray*}
\gamma_{2k-1}^{(n)} = 2^n\gamma_{k}^{(n-1)},\,\,\gamma_{2k}^{(n)} = \gamma_{k}^{(n-1)}.
\end{eqnarray*}
Hence
$
\gamma_{2k-1}^{(n)} = 2^n\gamma_{2k}^{(n)}
$
for all $k=1, \ldots, 2^{n-1}.$
\end{proof}

\begin{lemma}
The algebra isomorphism   $\Phi$ is discontinuous  in the measure topology.
\end{lemma}

\begin{proof}
For each $n\geq 2$ take a partial isometry
\begin{eqnarray*}
v_n &=& \sum\limits_{i=1}^{2^{n-1}}e^{(n)}_{2i-1, 2i}.
\end{eqnarray*}
By \eqref{inveran} we have that $a_n^{-1}=\sum\limits_{j=1}^{2^n}\frac{1}{\gamma_j^{(n)}}e_{j,j}^{(n)}.$ Using the last equality we obtain that
\begin{eqnarray*}
\Phi(v_n) &=& \Phi_n\left(\sum\limits_{i=1}^{2^{n-1}}e^{(n)}_{2i-1, 2i}\right)=
\sum\limits_{k=1}^{2^n}\gamma_k^{(n)}e_{k,k}^{(n)}\sum\limits_{i=1}^{2^{n-1}}e^{(n)}_{2i-1, 2i}\sum\limits_{j=1}^{2^n}\frac{1}{\gamma_j^{(n)}}e_{j,j}^{(n)}\\
&=& \sum\limits_{i=1}^{2^{n-1}}\frac{\gamma_{2i-1}^{(n)}}{\gamma_{2i}^{(n)}}e^{(n)}_{2i-1, 2i}\stackrel{\eqref{gamma}}=
2^n\sum\limits_{i=1}^{2^{n-1}}e^{(n)}_{2i-1, 2i}=2^nv_n.
\end{eqnarray*}
and therefore
\begin{eqnarray}\label{phiv}
\left|\Phi(v_n)\right| &=& 2^n\sum\limits_{i=1}^{2^{n-1}}e^{(n)}_{2i, 2i}.
\end{eqnarray}
Since $||v_n||_{\cM}=1$ for all $n,$ it follows that $\displaystyle 2^{-n}v_n \to 0$ in measure.
But \eqref{phiv} show that the sequence
$\displaystyle \left\{\Phi\left(2^{-n}v_n\right)\right\}$  does not converge to zero in measure, because
\begin{eqnarray*}
\tau\left(l(\Phi(v_n))\right) & = \tau\left(\sum\limits_{i=1}^{2^{n-1}}e^{(n)}_{2i, 2i}\right)=\frac{1}{2}.
\end{eqnarray*}
for all $n\in \mathbb{N}.$  This means that
$\Phi$ is discontinuous  in the measure topology.
\end{proof}

So, we have proved the following result.

\begin{theorem}\label{discon}
The algebra $\mathcal{R}_\infty$  admits an algebra automorphism, which is discontinuous in the measure topology.
\end{theorem}

\begin{remark}
 It is clear that Theorem~\ref{ringisomorphism} is an extension of \cite[Theorem 1.4]{AK2020} and they provide the automatic continuity  of ring isomorphisms in the measure topology. Thus the above Theorem~\ref{discon}  shows that the condition $\cM\subset \cA$ is essential in both of these theorems.
\end{remark}

\end{document}